\documentclass[12pt,oneside,reqno]{amsart} 
\usepackage{etex}
\usepackage[T2A]{fontenc}
\usepackage[cp1251]{inputenc}
\usepackage{pifont}
\usepackage[dvips]{graphicx}
\usepackage{amsmath,latexsym,amsthm,cmap,indentfirst,setspace,ccaption,amssymb,color,amscd,mathrsfs,hyperref}
\usepackage{array,enumitem,tikz}

\usepackage{fourier}

\usepackage{amssymb,lipsum}
\usepackage{float,xypic}
\usepackage{cmap,longtable}

\usepackage{geometry}
\usepackage{caption, subcaption}

\usepackage{verbatim}
\hypersetup{
	bookmarks=true,         
	unicode=true,          
	pdftoolbar=true,        
	pdfmenubar=true,        
	pdffitwindow=false,     
	pdfnewwindow=true,      
	colorlinks=true,       
	linkcolor=blue,          
	citecolor=blue,        
	filecolor=magenta,      
	urlcolor=cyan           
}

\DeclareMathOperator{\Pic}{Pic}

\DeclareMathOperator{\Aut}{Aut}

\DeclareMathOperator{\pr}{pr}

\DeclareMathOperator{\PGL}{PGL}
\DeclareMathOperator{\SU}{SU}
\DeclareMathOperator{\GL}{GL}
\DeclareMathOperator{\SL}{SL}

\DeclareMathOperator{\SO}{SO}

\DeclareMathOperator{\OO}{O}

\DeclareMathOperator{\Mov}{Mov}
\DeclareMathOperator{\Amp}{Amp}

\DeclareMathOperator{\Bir}{Bir}

\DeclareMathOperator{\J}{J}
\DeclareMathOperator{\B}{B}
\DeclareMathOperator{\Cone}{\mathscr C}
\DeclareMathOperator{\Tits}{\mathscr T}
\DeclareMathOperator{\DTits}{\mathscr D}
\DeclareMathOperator{\Hilb}{Hilb}
\DeclareMathOperator{\Diff}{Diff}
\DeclareMathOperator{\NS}{NS}
\DeclareMathOperator{\Isom}{Isom}
\DeclareMathOperator{\Stab}{Stab}
\DeclareMathOperator{\Mob}{M\"{o}b}
\DeclareMathOperator{\UC}{UC}
\DeclareMathOperator{\Symmetric}{Sym}
\DeclareMathOperator{\dist}{dist}

\DeclareMathOperator{\Vect}{\mathscr V}
\DeclareMathOperator{\Min}{Min}
\DeclareMathOperator{\Gr}{Gr}
\DeclareMathOperator{\Kah}{Kah}
\DeclareMathOperator{\Pos}{Pos}
\DeclareMathOperator{\Nef}{Nef}

\theoremstyle{plain}
\newtheorem*{theoremA}{Theorem A}
\newtheorem*{theoremB}{Theorem B}
\newtheorem*{theoremC}{Theorem C}

\newtheorem*{theorem*}{Theorem}
\newtheorem*{conditional}{Conditional Theorem}
\newtheorem{thm}{Theorem}[section]
\newtheorem{lem}[thm]{Lemma} 

\newtheorem{sled}[thm]{Corollary} 
\newtheorem{prop}[thm]{Proposition} 
\newtheorem{quest}[thm]{Question} 

\theoremstyle{definition}
\newtheorem{mydef}[thm]{Definition} 
\newtheorem{rem}[thm]{Remark}
\newtheorem{ex}[thm]{Example}

\makeatletter
\g@addto@macro{\endabstract}{\@setabstract}
\newcommand{\authorfootnotes}{\renewcommand\thefootnote{\@fnsymbol\c@footnote}}%
\makeatother

\begin{document}
	
	\begin{center}
		\LARGE 
		Automorphisms of hyperk\"{a}hler manifolds\\ and groups acting on CAT(0) spaces \par \bigskip
		
		\normalsize
		\authorfootnotes
		\large Nikon Kurnosov and Egor Yasinsky
        \footnote{The first author was supported within the framework of a subsidy granted to the HSE by the Government of the Russian Federation for the implementation of the Global Competitiveness Program. The second author acknowledges support by the Swiss National Science Foundation Grant ``Birational
        	transformations of threefolds'' 200020\_178807. Both authors are supported by RFBR grant 18-31-00146$\setminus$18.}

		
	\end{center}

\newcommand{\CC}{\mathbb C} 
\newcommand{\KK}{\mathbb K} 
\newcommand{\HH}{\mathbb H}
\newcommand{\DD}{\mathbb B} 
\newcommand{\FF}{\mathbb F}
\newcommand{\QQ}{\mathbb Q}
\newcommand{\RR}{\mathbb R}
\newcommand{\ZZ}{\mathbb Z}
\newcommand{\PP}{\mathbb P}
\newcommand{\kk}{\mathbb{k}}
\newcommand{\Sph}{\mathbb S}
\newcommand{\Torus}{\mathbb T}
\newcommand{\RP}{\mathbb{RP}}
\newcommand{\Sym}{\mathfrak S}
\newcommand{\Alt}{\mathfrak A}
\newcommand{\Dih}{\mathrm D}
\newcommand{\XC}{X_{\mathbb C}}
\newcommand{\Weyl}{\mathscr W}
\newcommand{\id}{\mathrm{id}}
\newcommand{\DP}{\mathcal{D}}
\newcommand{\CAT}{\mathrm D}
\newcommand{\GGG}{\mathcal G}
\newcommand{\Quad}{Q}
\newcommand{\Limits}{\mathrm C}
\newcommand{\OOO}{\mathscr O}
\newcommand{\ball}{\mathbb B}
\newcommand{\EE}{\mathbb E}
\newcommand{\UU}{\mathbb U}
\newcommand{\BB}{\mathbb B}
\newcommand{\qform}{\mathfrak q}
\newcommand{\XX}{\mathcal X}
\newcommand{\Hyp}{\mathrm{Hyp}}
\newcommand{\Per}{\mathrm{Per}}
\newcommand{\Teich}{\mathrm{Teich}}
\newcommand{\Monodromy}{\mathrm{Mon}}

	\newcommand*\conj[1]{\overline{#1}}
	
	\begin{abstract}
	We study groups of biholomorphic and bimeromorphic automorphisms of projective hyperk\"{a}hler manifolds. Using an action of these groups on some non-positively curved space, we immediately deduce many of their properties, including finite presentation, strong form of Tits' alternative, and some structural results about groups consisting of transformations with infinite order.
       
	\end{abstract}


\section{Introduction}

The purpose of this paper is to prove some boundedness results on biholomorphic and bimeromorphic automorphism groups of hyperk\"{a}hler manifolds using geometric group theory. Some of these facts should be known to experts, but their proofs (sometimes quite recent) have different nature. Our goal is to put them into the context of groups acting geometrically on CAT(0) spaces, and to explain some new and old statements from that point of view. The advantage of this approach is that a group $\Gamma$ acting on a CAT(0) space <<nicely>> (properly and cocompactly by isometries) automatically has a lot of good properties which are classically known to people doing metric geometry. For example, a consequence of \v{S}varc-Milnor lemma states that $\Gamma$ is finitely presented in this case. For more properties see Section \ref{section: metric}.


Throughout this note, we work over the complex number field $\CC$. By a {\it hyperk\"{a}hler manifold} we mean a compact simply-connected complex K\"{a}hler manifold $M$ having everywhere non-degenerate holomorphic 2-form $\omega_M$ such that $H^0(M,\Omega_M^2)=\CC\omega_M$. These manifolds play a very important role in classification of compact K\"{a}hler manifolds with vanishing first Chern class. The known examples include K3 surfaces, the Hilbert schemes $\Hilb^n(S)$ of 0-dimensional closed subschemes of length $n$ of a K3 surface $S$, generalized Kummer varieties, i.e. the kernels of the composition 
\[
\Hilb^{n}(T)\to\Symmetric^n{T}\overset{s}{\to} T,
\]
where $T$ is a complex torus and $s$ is the sum morphism, and O'Grady's two sporadic examples of dimension 6 and 10.

Let $M$ be a hyperk\"{a}hler manifold. In present note we are interested in groups of its biholomorphic and bimeromorphic automorphisms, $\Aut(M)$ and $\Bir(M)$ respectively. Recall that a classical Tits' alternative states that any finitely generated linear algebraic group over a field is either virtually solvable (i.e. has a solvable subgroup of finite index), or contains a non-abelian free group. Following \cite{Og} let us say that a group $G$ is {\it almost abelian of finite rank $r$} if there are a normal subgroup $G'\lhd G$ of finite index and a finite group $K$ which fit in the exact sequence 
\[
\id\to K\to G'\to\ZZ^r\to 0.
\]
Then one has the following analogue of Tits' alternative for hyperk\"{a}hler manifolds:
\begin{thm}[{\cite[Theorem 1.1]{Og}}]\label{thm: Oguiso Tits}
	Let $M$ be a projective hyperk\"{a}hler manifold and $G$ be a subgroup of $\Bir(M)$. Then $G$ satisfies either:
	\begin{enumerate}
		\item $G$ is an almost abelian group of finite rank, or
		\item $G$ contains a non-abelian free group.
	\end{enumerate}
\end{thm}

In the non-projective case the situation is even simpler:

\begin{thm}[{\cite[Theorem 1.5]{Og1}}]\label{thm: non-projective}
	Let $M$ be a non-projective hyperk\"{a}hler manifold. Then the group $\Bir(M)$ (and hence $\Aut(M)$) is almost abelian of rank at most $\max(\rho(M)-1,1)$. In particular, these groups are finitely generated.
\end{thm} 

In \cite{Og1} K. Oguiso also asked (Question 1.5) if the groups $\Bir(M)$ and $\Aut(M)$ are finitely generated for {\it projective} hyperk\"{a}hler manifolds. Using Global Torelli Theorem, S. Boissi\`ere and A. Sarti proved that $\Bir(M)$ is finitely generated. This does not imply that $\Aut (M)$ is finitely generated since $\Aut (M)$ is not necessarily of finite index in $\Bir (M)$. The question of finite generation of $\Aut(M)$ remained open until the recent preprint of Cattaneo and Fu \cite{Fu}, where the authors were able to give an affirmative answer to Oguiso's question. 

\begin{thm}[{\cite[Theorem 1.5]{Fu}}]\label{thm: Catanneo}
	Let $M$ be a projective hyperk\"{a}hler manifold. Then the group $\Aut(M)$ is finitely presented.
\end{thm}

%
The goal of this note is to show that many finiteness properties of automorphism groups of projective hyperk\"{a}hler manifolds (including Theorem \ref{thm: Catanneo} and Theorem \ref{thm: Oguiso Tits}) follow from the fact that these groups act geometrically on some metric space of non-positive curvature. These metric spaces are the so-called CAT(0) spaces. Roughly speaking, these are spaces which are at least as non-positively curved as the Euclidean plane. Our method is not new: its rough sketch can be found already in \cite{TotaroHyperbolic} (in the context of K3 surfaces), and then it was applied in \cite{Benzerga} for proving that rational algebraic surfaces with a structure of so-called klt Calabi-Yau pair have finitely many real forms. However, to apply the same strategy to hyperk\"{a}hler manifolds one needs to have some tools which were developed only recently (e.g. that the Kawamata-Morrison cone conjecture holds for hyperk\"{a}hler manifolds). On the other hand, the fact that finite presentation of the automorphism group of a projective hyperk\"{a}hler manifold was proved only in 2018, shows that connection of these manifolds with CAT(0) spaces is still not a common knowledge. Therefore, in this note we tried to give a reasonably self-contained account of the corresponding construction. Our first main result is the following:

\begin{theoremA}
	Let $M$ be a projective hyperk\"{a}hler manifold. Then the groups $\Aut(M)$ and $\Bir(M)$ are CAT(0) groups. In particular, they are finitely presented.
\end{theoremA}

From this theorem we easily deduce the following strong form of Tits' alternative for hyperk\"{a}hler manifolds.

\begin{theoremB}\label{thm: Our Tits main}
	Let $M$ be a projective hyperk\"{a}hler manifold, and $G\subseteq\Bir(M)$ be a subgroup. Then
	\begin{enumerate}
		\item either $G$ contains a finite index subgroup isomorphic to $\ZZ^n$;
		\item or $G$ contains a non-commutative free group.
	\end{enumerate}
	In particular, so are $\Bir(M)$ and $\Aut(M)$.
\end{theoremB}

\begin{theoremC}[cf. Theorem 7.1 in \cite{Fu}]\label{thm: main 2}
	Let $M$ be a projective hyperk\"{a}hler manifold. Then 
		the groups $\Aut(M)$ and $\Bir(M)$ have finitely many conjugacy classes of finite subgroups. In particular, there exists a constant\footnote{This statement also follows from \cite[Theorem 1.8]{PS-Jordan}, see Remark \ref{rem: Costya-Prokh} below.} $\B=\B(M)$ such that $|G|\leqslant \B$ for any finite subgroup $G\subset\Bir(M)$.
\end{theoremC}
\begin{rem}\label{rem: platonov}
		A reader familiar with some basic properties of hyperk\"{a}hler manifolds might have an impression that Theorem C easily follows from the fact that both $\Aut(M)$ and $\Bir(M)$ admit a natural representation in $\GL(\NS(M))$ with a finite kernel, and the groups $\GL_n(\ZZ)$ are known to have only finitely many conjugacy classes of finite subgroups. However, even finitely generated subgroups of $\GL_n(\ZZ)$ may violate the latter property. Indeed, in \cite{Platonov} Grunewald and Platonov give an example of a finitely generated subgroup of $\SL_4(\ZZ)$ that contains infinitely many conjugacy classes of elements of order 4.
\end{rem}

\begin{rem}
	Let $\mathcal{G}$ be a family of groups. Following the terminology introduced in \cite{Popov} (see also \cite{PS-Jordan,PS-Analytic}) we say that $\mathcal{G}$ is {\it uniformly Jordan} (resp. has {\it uniformly bounded finite subgroups}) if there is a constant $\J=\J(\mathcal{G})$ (resp. $\B=\B(\mathcal{G})$) such that for any group $\Gamma\in\mathcal{G}$ and any finite subgroup $G\subset\Gamma$ there exists a normal abelian subgroup $A\subset G$ of index at most $\J$ (resp. $|G|\leqslant\B$). We say that $\Gamma$ is {\it Jordan} (resp. has {\it bounded finite subgroups}) if the family $\{\Gamma\}$ is uniformly Jordan (resp. bounded). In view of Theorem C and Remark \ref{rem: K3 finite} it is natural then to ask if the following group-theoretic analog of Beauville's finiteness conjecture is true:
	\begin{quest}
		Consider the family
		\[
		\GGG_n=\big \{\Bir(M):\ M\ \text{is a projective hyperk\"{a}hler manifold of dimension}\ 2n\big \}.
		\]
		Does the familly $\GGG_n$ have uniformly bounded finite subgroups (with a constant $\B=\B(n)$ depending only on $n$)? Is it at least uniformly Jordan with $\J=\J(n)$? Same questions for $\Aut(M)$.
	\end{quest}
	In some particular cases one can hope to obtain such bounds using results of  \cite{Guan}, \cite{Nikon}, and \cite{Sawon}.
\end{rem}

This paper is organized as follows. In Section \ref{sec: prelim} we recall some basic facts about hyperk\"{a}hler manifolds, and then discuss the Kawamata-Morrison cone conjecture for these manifolds, proved by Amerik and Verbitsky in \cite{KM} for all possible values of second Betti number except $b_2=5$. We then indicate the proof of this conjecture in the case $b_2=5$, expanding the sketch of the proof from \cite{KM2}.  Section \ref{section: metric} is devoted to CAT(0) metric spaces and properties of groups, which act geometrically on them. We show how the groups of biholomorphic and bimeromorphic self-maps of projective hyperk\"{a}hler manifolds are related to these spaces, and then deduce Theorem B and Theorem C from Theorem A. We also obtain some structural results about groups consisting of transformations of infinite order. In Section \ref{sec: construction} we construct a CAT(0) space, on which the groups $\Aut(M)$ and $\Bir(M)$ act properly and cocompactly by isometries, hence proving Theorem A. Although the construction of the space is not difficult, one should verify many technical conditions which guarantee that our groups act <<nicely>> on the resulting space. This is why we divided the proof into several steps, which we explain in details in further subsections. Finally, in Section \ref{sec: examples} we discuss some further examples, where this technique could be applied. 

\section{Preliminaries}\label{sec: prelim}

\subsection{Hyperk\"{a}hler manifolds}

By a {\it hyperk\"{a}hler} (or {\it irreducible holomorphic symplectic}) manifold we mean a compact simply-connected complex K\"{a}hler manifold $M$ having everywhere non-degenerate holomorphic 2-form $\omega_M$ such that $H^0(M,\Omega_M^2)=\CC\omega_M$. These manifolds are even dimensional Calabi-Yau manifolds and play a very important role in classification of K\"{a}hler manifolds with trivial Chern class. Namely, Beauville-Bogomolov decomposition theorem (\cite{B}) states that for any compact K\"{a}hler manifold with trivial Chern class there exists a finite \'etale cover $\widetilde{M}\to M$ such that
\[
\widetilde{M}\cong T\times\prod_{i=1}^{n}Y_i\times\prod_{j=1}^{m}Z_j,
\]
where $T$ is a complex torus, $Y_i$ are strict Calabi-Yau manifolds (with $\pi_1(Y_i)=0$, $K_{Y_i}=\OOO_{Y_i}$ and $h^{0,p}=0$ for $0<p<\dim Y_i$), and $Z_j$ are hyperk\"{a}hler manifolds.

One of the most important properties of a hyperk\"{a}hler manifold is the existence of Beauville-Bogomolov-Fujiki form (BBF-form for short) (\cite{Bea, F}. This is an integral symmetric bilinear form $q_{BBF}$ on $H^2(M,\ZZ)$ of signature $(3, 0, b_2(M)-3)$. By means of BBF-form, the signature of the Neron-Severi group $\NS(M)$ is one of the following:
\[
(1, 0, \rho(M)-1),\ \ \ (0, 1, \rho(M)-1),\ \ (0, 0, \rho(M)).
\]
We call these three cases hyperbolic, parabolic, and elliptic respectively. Due to a deep result of Huybrechts, $M$ is projective if and only if $\NS(M)$ is hyperbolic \cite{H}, \cite[Proposition 26.13]{GHJ}. 

Let $f: M\dashrightarrow M'$ be a bimeromorphic map between Calabi-Yau manifolds. By \cite[III.25.14]{GHJ} this map $f$ is an isomorphism in codimension 1 and induces a linear isomorphism
\[
f^*: H^2(M',\ZZ)\overset{\sim}{\to} H^2(M,\ZZ)
\]
which in the case of hyperk\"{a}hler manifolds preserves the BBF-form. In particular, there is a group homomorphism
\[
\Phi_{\NS}: \Bir(M)\to\OO(\NS(M)\otimes\RR,q_{BBF}).
\]
Put 
\[
\Bir^*(M)=\Phi_{\NS}(\Bir(M)),\ \ \ \Aut^*(M)=\Phi_{\NS}(\Aut(M)).
\]
We have the following important fact.

\begin{prop}[{\cite[Proposition 2.4]{OgPic2}}]\label{prop: kernel is finite}
	Let $M$ be a projective Calabi-Yau manifold. Then the kernel of a homomorphism 
	\[
	\Phi_{\NS}: \Bir(M)\to\GL(\NS(M))
	\]
	is a finite group.
\end{prop}

\subsection{The Kawamata-Morrison conjecture}

In this note we shall consider various cones (i.e. subsets stable under multiplication by $\RR_{>0}$) in the finite-dimensional vector space $\NS(M)_\RR=\NS(M)\otimes\RR$ equipped with $\ZZ$-structure given by $\NS(M)=H^{1,1}(M,\RR)\cap H^2(M,\ZZ)$.

Let $M$ be a compact, K\"{a}hler manifold. In what follows $\Kah(M)\subset H^{1,1}(M,\RR)$ will denote the open convex {\it K\"{a}hler cone} of $M$. Its closure $\Nef(M)=\overline{\Kah(M)}$ in $H^{1,1}(M,\RR)$ is called the {\it nef cone}. Further, $\Amp(M)$ will denote the {\it ample cone} of $M$. For some varieties these cones have a nice structure, e.g. for Fano varieties they are rational polyhedral. However in general they can be quite mysterious: they can have infinitely many isolated extremal rays or <<round>> parts. Both phenomena occur already for K3 surfaces. The Kawamata-Morrison cone conjecture predicts that for Calabi-Yau varieties the structure of these cones (or rather some closely related cones) is nice ``up to the action of the automorphism group''. Before stating a suitable version of this conjecture, we need some definitions.

Let $V$ be a finite-dimensional real vector space equipped with a fixed $\QQ$-structure.  
A {\it rational polyhedral cone} in $V$ is a cone, which is an intersection of finitely many half spaces defined over $\QQ$. In particular, such a cone is convex and has finitely many faces. For an open convex cone $\Cone\subset V$ we denote by $\Cone^+$ the convex hull of $\overline{\Cone}\cap V(\QQ)$.

Let $\Gamma$ be a group acting on a topological space $X$. A {\it fundamental domain} for the action of $\Gamma$ is a connected open subset $D\subset X$ such that
\[
\bigcup_{\gamma\in\Gamma}\gamma\cdot \overline{D}=X,
\]
and the sets $\gamma\cdot D$ are pairwise disjoint. Let $X$ be a subset of a metric space $Y$ (typically $Y$ will be either Euclidean or hyperbolic $n$-space). A {\it side} of a convex subset $C\subset Y$ is a maximal nonempty convex subset of $\partial C$. A {\it polyhedron} in $Y$ is a nonempty closed convex subset whose collection of sides is locally finite. A {\it fundamental polyhedron} for the action of a discrete isometry group $\Gamma$ on $X$ is a convex polyhedron $D$ whose interior is a locally finite fundamental domain for $\Gamma$. Local finiteness means that for each point $x\in X$ there is an open neighborhood $U$ of $x$ such that $U$ meets only finitely many sets $\gamma \overline{D}$, $\gamma\in\Gamma$. Obviously, this also implies that every compact subset $K\subset X$ intersects only finitely many sets $\gamma \overline{D}$.

One of the versions of the Kawamata-Morrison cone conjecture says that the action of the automorphism group of a Calabi-Yau variety on the cone $\Amp(M)^+$ has a rational polyhedral fundamental domain. There is also a birational version for $\Bir(M)$ and $\Mov(M)^+$ respectively, where $\Mov(M)$ denotes the {\it movable cone}, i.e. the convex hull in $\NS(M)_\RR$ of all classes of movable line bundles on $M$. The conjecture has been proved for K3 surfaces by Sterk and Namikawa \cite{Sterk}, \cite{Namikawa} using the Torelli theorem of Piatetski-Shapiro and Shafarevich, and generalized later on 2-dimensional Calabi-Yau pairs by Totaro \cite{Totaro}. For projective hyperk\"{a}hler manifolds the following versions of the Kawamata-Morrison conjectures were recently proved by E. Markman, E. Amerik and M. Verbitsky:

\begin{thm}\label{thm: cone theorems}
	Let $M$ be a projective simple hyperk\"{a}hler manifold. Then 
	\begin{enumerate}
		\item \cite[Theorem 5.6]{KM} The group $\Aut(M)$ has a finite polyhedral fundamental domain on $\Amp(M)^+$.
		\item \cite[Theorem 6.25]{M} The group $\Bir(M)$ has a rational polyhedral fundamental domain on $\Mov(M)^+$.
	\end{enumerate}
\end{thm}

\begin{rem}\label{rem: domain properties}
	For the reader who would like to follow Ratcliffe's \cite{Rat} exposition of geometrically finite groups, while reading Section \ref{subsec: step 1}, it may be useful to keep in mind a somewhat more explicit construction of fundamental polyhedrons in Theorem \ref{thm: cone theorems}. This is due to E. Looijenga \cite[Proposition 4.1 and Application 4.14]{Loo}. Let $C$ be a non-degenerate open convex cone in a finite dimensional real vector space $V$ equipped with a fixed $\QQ$-structure. Let $\Gamma$ be a subgroup of $\GL(V)$ which stabilizes $C$ and some lattice in $V(\QQ)$. Assume that there exists a polyhedral cone $\Pi$ in $C^+$ such that $\Gamma\cdot\Pi\supseteq C$ and there is an element $\xi\in C^\circ\cap V^*(\QQ)$ whose stabilizer $\Gamma_\xi$ is trivial\footnote{In our situation $\xi$ exists automatically even without assuming that there is a fundamental domain for $\Gamma$, see e.g. \cite[Proposition 6.6]{Fu}} (here $C^\circ$ denotes the open dual cone of $C$). Then $\Gamma$ admits a rational polyhedral fundamental domain $\Sigma$ on $C^+$. Moreover, as was noticed before \cite[Theorem 3.1]{Totaro} (and proved in \cite[Lemma 2.2]{TotaroArxiv}) Looijenga's fundamental domain coincides with a {\it Dirichlet domain} of $\Gamma$ when the representation preserves a bilinear form of signature $(1,*)$. Recall that for a discontinuous group $\Gamma$ of isometries of a metric space $(X,d)$ and a point $\xi\in X$ with a trivial stabilizer $\Gamma_\xi$ one defines the Dirichlet domain for $\Gamma$ as the set
	\[
	D_\xi(\Gamma)=\big\{x\in X:\ d(x,\xi)\leqslant d(x,g\xi)\ \text{for all}\ g\in\Gamma\big\}.
	\]
	Dirichlet polyhedrons are known to have many good properties, in particular they are locally finite in the interior of the positive cone \cite[Corollary 2.3]{TotaroArxiv}) and exact \cite[Theorem 6.6.2]{Rat} (meaning that for each side $S$ of $D=D_\xi(\Gamma)$ there is an element $\gamma\in\Gamma$ such that $S=D\cap\gamma D$). 
\end{rem}


Theorem \ref{thm: cone theorems} (1) has been initially proved in an assumption $b_2 \neq 5$. Below we sketch a proof for the case $b_2 =5$, which follows from the results of Amerik and Verbitsky (see also \cite[Remark 1.5]{KM2}).

\begin{prop} \label{prop: b_2=5}
Let
$M$ be a projective hyperk\"ahler manifold with
$b_2 = 5$. The automorphism group has a rational polyhedral fundamental domain on the
ample cone of $M$.
\end{prop}

Recall that the {\it mapping class group} is the group $\Diff(M )/ \Diff_0 (M )$, where $\Diff_0 (M )$ is a connected component of diffeomorphism group of $M$ (the group of isotopies). Consider the subgroup of the mapping class group which fixes the connected component of our chosen complex structure. The {\it monodromy group} is the image of this subgroup in $\OO(H^2 (M, \mathbb{Z}))$.

Denote by $\Hyp$ an infinite-dimensional space of all quaternionic triples $I, J, K$ on $M$ which are induced by some hyperk\"ahler structure, with the same $C^\infty$-topology of convergence with all derivatives. Identify $\Hyp_m= \Hyp /\SU(2)$ with the space of
all hyperk\"ahler metrics of fixed volume. 

Define the {\it Teichm\"uller space} $\Teich_h$ of hyper\"ahler structures
as the quotient $\Hyp_m / \Diff_0$. Define the period space of hyperk\"ahler
structures by the space $\Per_h = Gr_{ +++} (H^2 (M, \RR))$ of all positive oriented 3-dimensional
subspaces in $H^2 (M, \RR)$.

\begin{rem}
The period space $\Per_h$ is naturally diffeomorphic to $\SO(b_2 -3, 3)/\SO(3)\times \SO(b_2 -3)$. The map $\mathcal{P}er_h : \Teich_h \to \Per_h$ is the period map associating the 3-dimensional
space generated by the three K\"ahler forms $\omega_I, \omega_J, \omega_K$ to a hyperk\"ahler structure
$(M, I, J, K, g)$. This map by \cite[Theorem 4.9]{KM4} is an open embedding
for each connected component. Moreover, its image is the set of all spaces
$W \in \Per_h$ such that the orthogonal complement $W^\perp$ contains no MBM classes (see below).
\end{rem}

A non-zero negative rational homology class (1, 1)-class $z$ is
called {\it monodromy birationally minimal (MBM)} if for some isometry  $\gamma \in O(H^2 (M, \mathbb{Z}))$ belonging to the monodromy group, $\gamma(z)^{\perp} \subset H^{1,1} (M )$ contains a
face of the pull-back of the K\"ahler cone of one of birational models $M'$ of $M$. 

\begin{proof}[Proof of Proposition \ref{prop: b_2=5}]

For each primitive MBM class $r$,
denote by $S_r$ the set of all 3-planes $W \in \Gr_{+++}$ orthogonal to $r$. Consider the union $\cup_r S_r$ of this sets. Its complement in $\Gr_{+++}$ is identified to a connected component of the Teichm\"uller space by \cite[Theorem 4.9]{KM4}. So it is open. From the \cite[Theorem 1.7]{KM2} for $X = G/K$, where
$G = \SO(3, 2)$ and $K = \SO(3) \times \SO(2)$ \footnote{In the general case $G = \SO(3, b_2-3)$ and $K = SO(3) \times \SO(b_2-3)$}  it follows that monodromy group acts on the set of MBM classes with finite number of orbits. 

Recall that the monodromy acts by isometries, thus the square of a primitive MBM class in respect with the Beauvile-Bogomolov-Fujiki form on $M$ is bounded in a absolute value. This is key assumption in Amerik-Verbitsky's proof of Kawamata-Morrison cone conjecture. Indeed, the \cite[Theorem 6.6]{KM3} implies the finitness of orbits for the K\"ahler cone.

Consider the quotient $S=(\Pos(M) \cap \NS(M) \otimes \mathbb{R}) /\Gamma$, where $\Pos(M)$ is positive cone and $\Gamma$ is the Hodge monodromy group. Then by Borel and
Harish-Chandra theorem $S$ is a complete hyperbolic manifold of
finite volume. Since $\Aut(M)$ acts with finite number of orbits on $\Kah(M)$, then the image of $\Amp(M )$ in $S$ is a hyperbolic manifold $T$ with finite boundary. One can prove that $T$ admits decomposition by finitely many cells with finite piecewise geodesic
boundary. Finite polyhedral fundamental domain on the
ample cone of $M$ is obtained by suitable liftings of this cells. We refer the reader to \cite[Theorem 5.6]{KM} for the further details.
\end{proof}


\section{CAT(0) groups: definitions and applications}\label{section: metric}

In this subsection we first recall some basic definitions of the theory of CAT(0) spaces, and then show how to use them to prove our results.

\subsection{\bf CAT(0) spaces}

Let $(X,\dist)$ be a metric space. Recall that a geodesic segment joining two points $x,y\in X$ is the image of a path of length $\dist(x,y)$ joining $x$ and $y$. We write $[x,y]$ to denote {\it some} geodesic segment (which need not to be unique). A metric space is said to be geodesic if every two poins in $X$ can be joined by a geodesic. A {\it geodesic triangle} in $X$ consists of three points $x,y,z\in X$ and a choice of geodesic segments $[x,y]$, $[y,z]$ and $[x,z]$. 

A geodesic metric space $(X,\dist)$ is said to be a {\it CAT(0) space} if for every geodesic triangle $\Delta\subset X$ there exists a triangle $\Delta'\subset\EE^n$ (here and throughout the paper $\EE^n$ denotes the Euclidean $n$-space with a standard metric) with sides of the same length as the sides of $\Delta$, such that distances between points on $\Delta$ are less or equal to the distances between corresponding points on $\Delta'$. Informally speaking, this means that geodesic triangles in $X$ are <<not thicker>> than Euclidean ones.

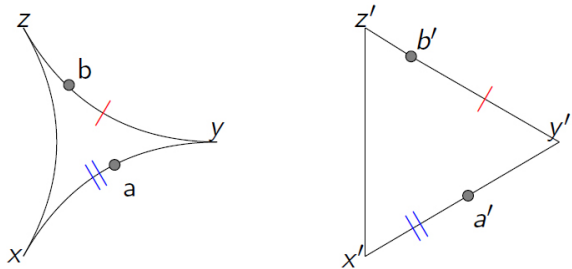
\begin{figure}[h!]
	\centering
	\includegraphics[width=0.5\linewidth]{triangles.jpg}
	\label{pic:triangles}
\end{figure}

\begin{mydef}[\bf CAT(0) groups]
	Let $\Gamma$ be a group acting by isometries on a metric space $X$. This action is {\it proper} or {\it properly discontinous} if for each $x\in X$ there exists $r>0$ such that the set of $\gamma\in\Gamma$ with
	\[
	\gamma\cdot B(x,r)\cap B(x,r)\ne\varnothing
	\]
	is finite (here and throughout the paper $B(x,r)$ denotes an open ball with center $x$ and radius $r$). The action is {\it cocompact} is there exists a compact set $K\subset X$ such that $X=\Gamma\cdot K$. The action is called {\it geometric} if it is proper and cocompact. Finally, we say that $\Gamma$ is a {\it CAT(0) group} if it acts geometrically on a CAT(0) space.
\end{mydef}

Now we can state one of the main results of this paper.

\begin{thm}[Theorem A]\label{thm: hyperkahler are CAT(0)}
	Let $M$ be a projective hyperk\"{a}hler manifold. Then $\Aut^*(M)$ and $\Bir^*(M)$ are CAT(0) groups. The same holds for $\Aut(M)$ and $\Bir(M)$.
\end{thm}
\begin{proof}
	See after Theorem \ref{thm: main technical}.
\end{proof}

The following properties of CAT(0) groups will be crucial for us.

\begin{thm}[{\cite[III.$\Gamma$, Theorem 1.1]{BH}}]\label{thm: finiteness properties}
	Every CAT(0) group $\Gamma$ satisfies the following properties:
	\begin{enumerate}
		\item $\Gamma$ finitely presented;
		\item $\Gamma$ has finitely many conjugacy classes of finite subgroups;
		\item Every solvable subgroup of $\Gamma$ has an abelian subgroup of finite index;
		\item Every abelian subgroup of $\Gamma$ is finitely generated.
	\end{enumerate}
\end{thm}

We postpone the construction of a CAT(0) space for $\Aut(M)$ and $\Bir(M)$ until Section \ref{sec: construction}. Below we deduce some corollaries from Theorem \ref{thm: hyperkahler are CAT(0)}.

\begin{sled}[Theorem A]
	Let $M$ be a hyperk\"{a}hler manifold. Then the groups $\Aut(M)$ and $\Bir(M)$ are finitely presented.
\end{sled}
\begin{proof}
	If $M$ is non-projective, then the groups $\Aut(M)$ and $\Bir(M)$ are almost abelian by Theorem \ref{thm: non-projective}, hence finitely presented. For projective hyperk\"{a}hler manifolds the statement follows from Theorem \ref{thm: finiteness properties} (1). 
\end{proof}

\begin{sled}[Theorem C]\label{sled: BFS}
	Let $M$ be a projective hyperk\"{a}hler manifold. Then the groups $\Aut(M)$ and $\Bir(M)$ have finitely many conjugacy classes of finite subgroups. In particular, there exists a constant $B=B(M)$ such that for every finite subgroup $G\subset \Bir(M)$ one has $|G|\leqslant B$ (i.e. $\Bir(M)$ has bounded finite subgroups).
\end{sled}
\begin{proof}
	Follows from Theorem \ref{thm: finiteness properties} (2).
\end{proof}

Of course, the second part of this statement (as well as Corollary \ref{sled: Burnside} below) can be obtained using Minkowski's Theorem, which states that $\GL_n(\QQ)$ has bounded finite subgroups (see e.g. \cite{Serre}). However, as we mentioned in Remark \ref{rem: platonov} the finiteness of conjugacy classes of finite subgroups is a much more subtle issue.

\begin{rem}\label{rem: Costya-Prokh}
	One can compare this result with \cite[Theorem 1.8]{PS-Jordan} which states that $\Bir(X)$ has bounded finite subgroups provided that $X$ is an irreducible algebraic variety which is non-uniruled and has $h^1(X,\OOO_X)=0$. The latter condition is clearly true for any projective hyperk\"{a}hler manifold $X$, and the former one holds since complex uniruled varieties have Kodaira dimension $-\infty$, while  Calabi-Yau manifolds have Kodaira dimension zero.
\end{rem}

\begin{rem}\label{rem: K3 finite}
	In dimension two, i.e. for projective K3 surfaces, it is known that the orders of their finite automorphism groups are bounded by $3840$ (and this bound is sharp) \cite{Kondo}. 
	
\end{rem}


Recall that a torsion group is a group in which each element has finite order. In general it is an open question whether torsion subgroups of any CAT(0) group are always finite. However in our case the answer to this question is positive. 

\begin{sled}[Burnside property]\label{sled: Burnside}
	Let $M$ be a projective hyperk\"{a}hler manifold. Then every 
	torsion subgroup $G\subseteq\Bir(M)$ is finite.
\end{sled}
\begin{proof}
	Put $G^*=\Phi_{\NS}(G)$ and $G_0=G\cap\ker\Phi_{\NS}$. One has a short exact sequence
	\[
	1\to G_0\to G\to G^*\to 1
	\]
	with $G_0$ finite, and $G^*$ a torsion group. By Theorem \ref{thm: finiteness properties} (2) (or Corollary \ref{sled: BFS}) the group $G^*$ has bounded exponent, i.e. there exists $d\in\ZZ_{>0}$ such that the order of any $g\in G^*$ is $\leqslant d$. Since $G^*$ is linear, it must be finite by Burnside's theorem. Therefore $G$ is finite too.
\end{proof}

\subsection{Tits' alternative}

In this subsection we show how our method implies a strong form of Tits' alternative (Theorem B) for projective hyperk\"{a}hler manifolds. In general this is a well-known open question whether CAT(0) groups always satisfy Tits' alternative, but in our case the usual Tits' alternative for $\GL_n(\QQ)$ and some properties of CAT(0) groups give even stronger restrictions than in the classical settings.

The heart of the proof of \cite[Theorem 1.1]{Og} was the fact that a virtually solvable subgroup of $\OO(L)$, where $L$ is a hyperbolic lattice of finite rank, must be almost abelian of finite rank. The proof of the latter involves Lie-Kolchin Theorem and various properties of Salem polynomials. In our case the key ingredient of Oguiso's proof follows from the fact that $\Bir(M)$ is a CAT(0) group. But in fact we are able to prove something stronger, namely, that in the first case of Tits alternative our group is just $\ZZ^n$ up to finite index.

\begin{thm}[Theorem B]\label{thm: Tits text}
	Let $M$ be a projective hyperk\"{a}hler manifold, and $G\subseteq\Bir(M)$ be a subgroup. Then
	\begin{enumerate}
		\item either $G$ contains a finite index subgroup isomorphic to $\ZZ^n$;
		\item or $G$ contains a non-commutative free group.
	\end{enumerate}
In particular, so are $\Bir(M)$ and $\Aut(M)$.
\end{thm}
\begin{proof}
	Put $G^*=\Phi_{\NS}(G)$. Then one has a short exact sequence of groups
	\[
	1\to N\to G\to G^*\to 1,
	\]
	with $N$ being a finite group by Proposition \ref{prop: kernel is finite}. Assume that $G^*$ does not contain a non-abelian
	free subgroup. Then by usual Tits' alternative for $\GL(\NS(M)\otimes\RR)$ the group $G^*$ has a solvable subgroup $S^*$ of finite index. Put $S=\Phi^{-1}_{\NS}(S^*)$. We have a short exact sequence
	\[
	1\to N\to S\to S^*\to 1
	\]
	with $[G:S]<\infty$, $N$ finite, and $S^*$ solvable. The centralizer $C=C_{S}(N)$ of $N$ in $S$ has finite index in $S$ (indeed, $S$ acts on $N$ by conjugation, which gives a homomorphism $S\to\Aut(N)$ with kernel $C_{S}(N)$ and $\Aut(N)$ a finite group). Thus we have an extension
	\[
	1\to A\to C\to C^*\to 1
	\]
	with $A=N\cap C$ abelian and $C^*$ solvable group. Clearly $[G:C]<\infty$. Since both $A$ and $C^*$ are solvable, the group $C\subset\Bir(M)$ is solvable.  By Theorem \ref{thm: finiteness properties} (3) and (4) it then contains $F\cong\ZZ^n$ with $[C:F]<\infty$. Hence $G$ contains a finite index subgroup isomorphic to $\ZZ^n$.
\end{proof}

\subsection{Some applications to dynamics}

Let $(X,\dist)$ be a metric space and $f\in\Isom(X)$ be its isometry. Then one can consider the {\it displacement function} of $f$
\[
d_f: X\to\RR_{\geqslant 0},\ \ \ d_f(x)=\dist(f(x),x).
\]
The {\it translation length} of $f$ is the number $\|f\|=\inf\{d_f(x):\ x\in X\}$. The set of points where $d_f$ attains the infimum is denoted by $\Min(f)$. If $d_f$ attains a strictly positive minimum, then $f$ is called {\it loxodromic}; if this minimum is 0 (i.e. $f$ has a fixed point), then $f$ is called {\it elliptic}; if $d_f$ does not attain the minimum (i.e. $\Min(f)=\varnothing$), then $f$ is called {\it parabolic}. Elliptic and loxodromic isometries are also called {\it semi-simple}. In the case $X=\HH^n$ these definitions agree with the old ones.  

Now let $M$ be a projective hyperk\"{a}hler manifold, and $f\in\Bir(M)$ be birational automorphism. According to the action of $f^*$ on the corresponding hyperbolic space $(\NS(M)_\RR,q_{BBF})$ one can classify $f$ as elliptic, parabolic or loxodromic. Denote by $\XX_M$ the CAT(0) space constructed in Section \ref{sec: construction}, i.e. the space on which $\Bir(M)$ acts properly and cocompactly by isometries of $\XX_M$. Then one has a group homomorphism
\[
\Theta_M: \Bir(M)\to\Isom(\XX_M).
\]
We should warn the reader that in general $\Theta_M$ does not preserve\footnote{It is not surprising since the metric on $\XX_M$ is not the same as on the initial hyperbolic space.} the type of an isometry. In fact, $\Theta(\Bir(M))$ does not contain parabolic isometries by \cite[II.6.10 (2)]{BH}. 

\begin{lem}\label{lem: preserving of types}
	The images of $M$-loxodromic and $M$-parabolic birational automorphisms under $\Theta_M$ are $\XX_M$-loxodromic. The images of $M$-elliptic elements are $\XX_M$-elliptic.  In particular, $\Theta_M$ maps semi-simple isometries to semi-simple ones.
\end{lem}
\begin{proof}
	First note that if a group $\Gamma$ acts geometrically on a proper CAT(0) space then $\gamma\in\Gamma$ has finite order if and only if $\gamma$ is elliptic. 
	
	Let $f\in\Bir(M)$ be of infinite order, i.e. either $M$-loxodromic or $M$-parabolic. Then, as was noticed above, $\Theta_M(f)$ is either $\XX_M$-loxodromic, or $\XX_M$-elliptic. In the latter case $\Theta_M(f)^n=\id$ for some $n>0$. Thus $f^n\in\ker\Theta_M$, i.e. $f^n$ acts as identity on $\XX_M$. But this also means that $f$ has a fixed point locus on the underlying hyperbolic space $(\NS(M)_\RR,q_{BBF})$, i.e. $f^n$ is $M$-elliptic. By \cite[II.6.7]{BH} we have that $f$ must be $M$-elliptic too, contradiction. Finally, the image of an element of finite order is of finite order, hence $M$-elliptic elements map to $\XX_M$-elliptic elements.
\end{proof}

Below we shall use the following important {\bf Flat Torus Theorem}.
\begin{theorem*}[{\cite[II.7.1]{BH}}]
	Let $A$ be a free abelian group of rank $n$ acting properly by semi-simple isometries on a CAT(0) space $X$. Then:
	\begin{enumerate}
		\item $\Min(A)=\cap_{\alpha\in A}\Min(\alpha)$ is non-empty and splits as a product $Y\times\EE^n$;
		\item Every element $\alpha\in A$ leaves $\Min(A)$ invariant and respects the product decomposition; $\alpha$ acts as the identity on $Y$ and as a translation on $\EE^n$;
		\item If a finitely generated subgroup $\Gamma\subset\Isom(X)$ normalizes $A$, then $\Gamma$ has a subgroup of finite index that contains $A$ as a direct factor.
	\end{enumerate}
\end{theorem*}

\begin{prop}
	Let $M$ be a projective hyperk\"{a}hler manifold, and $f\in\Bir(M)$ be either parabolic, or loxodromic. Denote by $C(f)$ the centralizer of $f$ in $\Bir(M)$. Then $C(f)$ has a finite index subgroup $H$ which splits as a direct product: $H=N\times\langle f\rangle$. 
\end{prop}
\begin{proof}
	Suppose that a group $\Gamma$ acts geometrically on a CAT(0) space $X$, and $\gamma$ is an element of infinite order. Then its centralizer $C(\gamma)$ acts geometrically on the CAT(0) subset $\Min(\gamma)$ of $X$ \cite[Theorem 3.2]{Ruane}. Thus we see that $C(f)$ is a CAT(0) group, hence finitely generated by Theorem \ref{thm: finiteness properties} (1). It remains to apply the Flat Torus Theorem (3) to $A=\langle f\rangle$. 
\end{proof}

Given a finitely generated group and its arbitrary element, it is natural to ask how the iterates of this element behave with respect of generators. Namely, let $\Gamma$ be a finitely generated group with finite symmetric generating set $\Sigma=\Sigma^{-1}$. Recall that the word metric on $\Gamma$ is defined as
\[
w_\Sigma(\gamma_1,\gamma_2)=\min\{n:\ \gamma_1^{-1}\gamma_2=\sigma_1\sigma_2\ldots\sigma_n,\ \sigma_i\in\Sigma\},
\]
and the length of $\gamma\in\Gamma$ is $|\gamma|_\Sigma=w_\Sigma(\id,\gamma)$. An element $\gamma\in\Gamma$ is called {\it distorted} if 
\[
\lim_{n\to\infty}\frac{|\gamma^n|_\Sigma}{n}=0
\]
and {\it undistorted} otherwise. The property of being undistorted is well known to be independent of choice of $\Sigma$.

\begin{prop}
	Let $M$ be a projective hyperk\"{a}hler manifold. Then its loxodromic and parabolic birational automorphisms are undistorted. 
\end{prop}
\begin{proof}
	Let $\gamma\in\Bir(M)$ be a $M$-loxodromic or $M$-parabolic automorphism. Then $\Theta(\gamma)=\Theta_M(\gamma)$ is of infinite order by Lemma \ref{lem: preserving of types}. By \cite[I.8.18]{BH} for any choice of basepoint $x_0\in\XX_M$ there exists a constant $\mu>0$ such that
	\[
	\dist_{\XX_M}(\gamma_1 x_0,\gamma_2 x_0)\leq \mu w_\Sigma(\gamma_1,\gamma_2).
	\]
	Then one has
	\begin{equation}\label{eq: distortion 1}
	\lim_{n\to\infty}\frac{|\gamma^n|_\Sigma}{n}\geqslant\lim_{n\to\infty}\frac{|\Theta(\gamma)^n|_{\Theta(\Sigma)}}{n}=\lim_{n\to\infty}\frac{w_{\Theta(\Sigma)}(\id,\Theta(\gamma)^n)}{n}\geqslant\lim_{n\to\infty}\frac{\mu^{-1}\dist_{\XX_M}(x_0,\Theta(\gamma)^nx_0)}{n},
	\end{equation}
	where $x_0\in\XX_M$ is an arbitrary point. Now let $X$ be a CAT(0) space, $\delta$ be a semi-simple isometry, and $x\in X$ be any point. Then it is easy to check that
	\begin{equation}\label{eq: distortion 2}
	\|\delta\|=\frac{\dist_X(x,\delta^n x)}{n}.
	\end{equation}
	By the Flat Torus Theorem, the set $\Min(\langle\delta\rangle)\equiv\cap_k\Min(\gamma^k)$ is $\gamma$-invariant ans splits as a product $Y\times\EE^1$ such that $\delta$ acts identically on $Y$ and by translations on $\EE^1$. It then easily follows that $\|\delta^n\|=n\cdot\|\delta\|$.  Now taking $X=\XX_M$ and $\delta=\Theta(\gamma)$ we get from (\ref{eq: distortion 1}) and (\ref{eq: distortion 2}) that
	\[
	\lim_{n\to\infty}\frac{|\gamma^n|_\Sigma}{n}\geqslant\mu^{-1}\|\Theta(\gamma)^n\|=\mu^{-1}n\|\Theta(\gamma)\|>0,
	\]
	since $\|\Theta(\gamma)\|>0$.
\end{proof}

\subsection{Cohomological properties}

Finally we would like to show that cohomological properties of $\Bir(M)$ and $\Aut(M)$ mentioned in \cite{Fu} can be also obtained using our approach. Recall that a group $\Gamma$ is called {\it of type FL} if the trivial $\ZZ[\Gamma]$-module $\ZZ$ has a finite resolution by free $\ZZ[\Gamma]$-modules of finite rank:
\[
0\to \ZZ[\Gamma]^{n_k}\to\ldots\to \ZZ[\Gamma]^{n_1}\to\ZZ\to 0.
\]
We say that $\Gamma$ is {\it of type VFL\footnote{Virtuellement une r\'{e}solution Libre de type Finie}} if it is virtually FL, i.e. admits a finite-index subgroup satisfying property FL.

\begin{prop}
	Let $M$ be a projective hyperk\"{a}hler manifold. Then the groups $\Aut(M)$ and $\Bir(M)$ are of type VFL.
\end{prop}
\begin{proof}
	Let $\Gamma$ denote either $\Aut(M)$ or $\Bir(M)$. By Selberg's lemma, the group $\Phi_{\NS}(\Gamma)$ is virtually torsion-free. Since $\ker\Phi_{\NS}$ is finite, $\Gamma$ aslo contains a finite-index torsion-free subgroup, say $\Gamma_0$. Consider the action of $\Gamma$ on the associated CAT(0) space $\XX_M$. Note that cocompactness is inherited under restriction of the action of $\Gamma$ to any finite-index subgroup, and properness holds for any subgroup of $\Gamma$. So, the action of $\Gamma_0$ on $\XX_M$ is proper and cocompact (and free). By \cite[III.$\Gamma$.1.1, II.5.13]{BH}, $\Gamma_0$	has a finite CW complex as classifying space. By \cite[VIII.6.3]{Brown} $\Gamma_0$ is of type FL then.
\end{proof}

We refer to \cite{Brown} for further finiteness properties of groups. 

\section{Constrution of a CAT(0) space}\label{sec: construction}

To prove Theorem \ref{thm: hyperkahler are CAT(0)} (and Theorem A) for a projective\footnote{We mention once again that in non-projective setting all our statements follow from Oguiso's work \cite{Og}, \cite{Og1}.} hyperk\"{a}hler manifold $M$ we shall construct a CAT(0) space where the groups $\Bir(M)$ and $\Aut(M)$ act properly and cocompactly by isometries. As was mentioned in Introduction, here we follow the ideas sketched in \cite{TotaroHyperbolic} (see also \cite{Benzerga} for more details), but try to make our exposition accessible for a non-expert in metric geometry. The main reference where the reader can find most technical facts used here is \cite{Rat} and \cite{Apanasov}. Our construction will involve a hyperbolic space, so we first recall some basic definitions.

\subsection{\bf Hyperbolic space and its isometries}\label{subsec: hyperbolic prelim} A hyperbolic $n$-dimensional space is an $n$-dimensional Riemannian simply connected space of constant negative curvature. Throughout this note we use several models of hyperbolic space, which we briefly describe below to establish notation. The main reference is \cite{Rat}. A reader with a good background in hyperbolic geometry can easily skip this subsection and go to Setup \ref{subsec: setup}.

\subsubsection{\bf Standard models} Let $\Vect$ be a Minkowski vector space of dimension $n+1$ with the quadratic form $\qform: \Vect\to\RR$ of signature $(1,n)$ and inner product of two vectors $v_1$, $v_2$ denoted by $\langle v_1,v_2\rangle$. We will choose the coordinates $x_0,\ldots, x_n$ in $\Vect$ such that $\qform=x_0^2-x_1^2-\ldots-x_n^2$. The vectors $v\in V$ with $\qform(v)=1$ form an n-dimensional hyperboloid consisting of two connected components: $H^{+}=\{x_0>0\}$, and $H^{-}=\{x_0<0\}$. The points of the {\it hyperboloid model} $\HH^n$ are the points on $H^+$. The distance function is given by $\dist(u,v)={\rm argcosh}\langle u,v\rangle$. The $m$-planes are represented by the intersections of the $(m+1)$-planes in $\Vect$ with $H^{+}$. 

The {\it Poincar\'{e} model} of $\HH^n$ has its points lying inside the unit open disk
\[
\BB^n=\{(0,x_1,\ldots,x_n)\in\RR^{n+1}:\ x_1^2+\ldots+x_n^2<1\}
\]
and is obtained from the hyperboloid model by means of the stereographic projection $\zeta$ from the south pole of the
unit sphere in $\Vect$ (i.e. the point $(-1,0,\ldots,0)$) on the hyperplane $V_0=\{x_0=0\}$. A subset $P\subset\BB^n$ is called a hyperbolic $m$-plane if and only if $\zeta(P)$ is a hyperbolic $m$-plane of $\HH^n$. 

Denote by $\EE^n$ the Euclidean $n$-space with standard Euclidean metric. By $\widehat{\EE}^n=\EE^n\cup\{\infty\}$ we denote its one-point compactification (e.g. $\widehat{\CC}$ is the Riemann sphere). If $a$ is a unit vector in $\EE^n$ and $r\in\RR$, then $P(a,r)$ is the hyperplane with unit normal vector $a$ passing through the point $ra$. Further, $S(a,r)$ denotes the sphere of radius $r$ centered at $a$. We shall also consider extended planes $\widehat{P}(a,r)=P(a,r)\cup\{\infty\}$. By a {\it sphere} in $\widehat{\EE}^n$ we mean either a Euclidean sphere or an extended plane (so, topologically a sphere too).

A $p$-sphere and a $q$-sphere of $\widehat{\EE}^n$ are said to be orthogonal if they intersect and at each finite point of intersection their tangent planes are orthogonal. One can show that $P\subset\BB^n$ is a hyperbolic $m$-plane of $\BB^n$ if and only if $P$ is the intersection of $\BB^n$ either with an $m$-dimensional vector subspace of $V_0=\EE^n$, or an $m$-sphere of $V_0$ orthogonal to $\partial\overline{\BB}^n$. 

In the Poincare ball model $\BB^n$, a {\it horoball} based at $a\in\partial\DD^n$ is an Euclidean ball contained in $\overline{\BB}^n$ which is tangent to $\partial\BB^n$ at the point $a$. Assume $\Gamma$ contains a parabolic element having $a\in\partial \BB^n$ as its fixed point. A {\it horocusp region} is an open horoball $B$ based at a point $a\in \partial \BB^n$ such that for all $\gamma\in\Gamma\setminus\Stab_\Gamma(a)$ one has $\gamma(B)\cap B=\varnothing$.

Finally, we mention the {\it upper half-space model}
\[
\UU^n=\{(x_1,\ldots,x_n)\in\EE^n:\ x_n>0\}
\]
with a metric induced from $\BB^n$ in the following way. Let $\sigma$ be the reflection of $\widehat{\EE}^n$ in the sphere $S(e_n,\sqrt{2})$ and $\rho$ be the reflection of $\widehat{\EE}^n$ in $\widehat{\EE}^{n-1}$. Then $\eta=\sigma\circ\rho$ maps homeomorphically $\UU^n$ to $\BB^n$. Put $\dist_{\UU^n}(u,v)=\dist_{\BB^n}(\eta(u),\eta(v))$. A subset $P\subset\UU^n$ is called a hyperbolic $m$-plane if and only if $\eta(P)$ is a hyperbolic $m$-plane of $\BB^n$. One can show that $P\subset\UU^n$ is a hyperbolic $m$-plane of $\UU^n$ if and only if $P$ is the intersection of $\UU^n$ either with an $m$-plane of $\EE^n$ orthogonal to $\EE^{n-1}$, or an $m$-sphere of $\EE^n$ orthogonal to $\EE^{n-1}$.

Recall that a {\it geodesic line} (or just {\it geodesic}) in a Riemannian manifold $M$ is a continuous map $\gamma: \RR\to M$ such that $\dist_M(\gamma(x),\gamma(y))=|x-y|$. We also refer to the image of $\gamma$ as a geodesic line. For any two distinct points $x,y\in M$ there exists a closed interval $[a; b]\subset \RR$ and a geodesic $\gamma$ with $\gamma(a)=x$, $\gamma(b)=y$, which is called the {\it geodesic segment}. In all described models of hyperbolic space, its geodesics are just hyperbolic lines, i.e. $1$-planes.

\subsubsection{\bf Isometries} For the reader's convenience, we shall use the language of M\"{o}bius transformation as in our main reference \cite{Rat} when talking about isometries of hyperbolic space.

A {\it M\"{o}bius transformation} of $\widehat{\EE}^n$ is a finite composition of reflections of $\widehat{\EE}^n$ in spheres. Consider $\EE^{n-1}\equiv\EE^{n-1}\times\{0\}\subset\EE^n$. Any $f\in\Mob(\widehat{\EE}^{n-1})$ can be extended to an element of $\Mob(\widehat{\EE}^n)$ as follows. If $f$ is a reflection of $\widehat{\EE}^{n-1}$ in $\widehat{P}(a,r)$ then $\widetilde{f}$ is the reflection of $\widehat{\EE}^n$ in $\widehat{P}(\widetilde{a},r)$ where $\widetilde{a}=(a,0)$. If $f$ is a reflection of $\widehat{\EE}^{n-1}$ in $S(a,r)$ then $\widetilde{f}$ is the reflection of $\widehat{\EE}^n$ in $S(\widetilde{a},r)$. The {\it  Poincar\'{e} extension} of an arbitrary $f=f_1\circ\ldots\circ f_m\in\Mob(\widehat{\EE}^{n-1})$ is then defined\footnote{It is easy to see that the definition is correct.} as $\widetilde{f}=\widetilde{f}_1\circ\ldots\circ \widetilde{f}_m$.

If $Y=\UU^n$ or $\BB^n$, a {\it M\"{o}bius transformation} $f\in\Mob(Y)$ is a M\"{o}bius transformation of $\widehat{\EE}^n$ that leaves $Y$ invariant. The element $f\in\Mob(\UU^n)$ is a M\"{o}bius transformation if and only if it is the Poincar\'{e} extension of an element of $\Mob(\widehat{\EE}^{n-1})$, so $\Mob(\UU^n)\cong\Mob(\widehat{\EE}^{n-1})$ \cite[\S 4.4]{Rat}. Similar statement holds for $\Mob(\BB^n)$.

Every M\"{o}bius transformation of $\BB^n$ restricts to an isometry of the conformal ball model $\BB^n$, and every isometry of $\BB^n$ extends to a unique M\"{o}bius transformation of $\BB^n$ \cite[Theorem 4.5.2]{Rat}. In particular, $\Isom(\BB^n)\cong\Mob(\BB^n)$.

Let $f\in\Mob(\BB^n)$ be a M\"{o}bius transformation (an isometry of the hyperbolic $n$-space). Then $f$ maps $\overline{\BB}^n$ into itself and by the Brouwer fixed point theorem, $f$ has a fixed point in $\overline{\DD}^n$. Recall that $f$ is said to be {\it elliptic} if $f$ fixes a point of $\BB^n$; {\it parabolic} if $f$ fixes no point of $\BB^n$ and fixes a unique point of $\partial\overline{\BB}^n=\Sph^{n-1}$; {\it loxodromic} if $f$ fixes no point of $\BB^n$ and fixes two points of $\Sph^{n-1}$, say $a$ and $b$. A hyperbolic line $L$ joining $a$ and $b$ is called the {\it axis} of $f$, and $f$ acts as a translation along $L$. If $f$ translates $L$ in the direction of $a$, then for any $x\in\overline{\BB}^n$, $x\ne b$, one has $f^m(x)\to a$ as $m\to\infty$, i.e. $a$ is an {\it attractive fixed point} (and $b$ is {\it repulsive}).

\subsection{Setup}\label{subsec: setup}

Assume that we have:

\begin{itemize}
	\item A vector space $\Vect$ of dimension $n+1$, $n\geqslant 2$, with a fixed $\ZZ$-structure $\Lambda\cong\ZZ^{n+1}$, $\Lambda\otimes\RR=\Vect$;
	\item A quadratic form $\qform: \Vect\to\RR$ of signature $(1,n)$;
	\item Hyperbolic space (with models $\HH^n$, $\BB^n$ and $\UU^n$), associated to $(\Vect,\qform)$;
	\item A convex cone $\Cone$ in $\Vect$;
	\item A group action $\Phi:\Gamma\to\GL(\Lambda)\subset\GL(\Vect)$ with discrete image $\Phi(\Gamma)$ and finite kernel, such that $\Gamma$ preserves $\Cone,\ \qform$, and has a rational polyhedral (locally finite) fundamental domain $\Omega$ on $\Cone$. 
\end{itemize}

\begin{thm}\label{thm: main technical}
	Suppose that all the conditions in Setup \ref{subsec: setup} hold. Then $\Gamma$ is a CAT(0) group. 
\end{thm}
\begin{proof}[Proof of Theorem \ref{thm: hyperkahler are CAT(0)} (Theorem A)]
	For $\rho(M)\geqslant 2$ it follows with $\Lambda=\NS(M)$, $\qform=q_{BBF}$, $\Cone=\Amp(M)^+$ for $\Gamma=\Aut(M)$ and $\Cone=\Mov(M)^{+}$ for $\Gamma=\Bir(M)$, and $\Omega$ given by Theorem \ref{thm: cone theorems} and Remark \ref{rem: domain properties}. Note that the kernel of $\Phi=\Phi_{\NS}$ is finite by Proposition \ref{prop: kernel is finite}.
	
	Now let us treat the case $\rho(M)=2$ separately (as will be clear from below, case $\rho(M)=1$ is easier). Let $\Gamma$ be either $\Aut(M)$ or $\Bir(M)$. By Proposition \ref{prop: kernel is finite} there is a short exact sequence
	\[
	1\to K\to \Gamma\overset{\Phi_{\NS}}{\longrightarrow} \Gamma^*\to 1
	\]
	with $\Gamma^*=\Phi_{\NS}(\Gamma)$ and $K$ a finite group. For any $g\in\Gamma^*$ one has $\det{g}=\pm 1$. Put 
	$
	\Gamma^+=\{g\in\Gamma^*:\ \det{g}=1\}.
	$
	By \cite[Theorem 3.9]{Lazic} one has either $\Gamma^+=1$ or $\Gamma^+\cong\ZZ$. It suffices to consider only the last case. Here we have a short exact sequence
	\[
	1\to K\to\Gamma'=\Phi_{\NS}^{-1}(\Gamma^+)\to\Gamma^+\to 1.
	\]
	It is well known (and easy to show) that a finite-by-cyclic group is always virtually cyclic. This means that $\Gamma'$, and hence $\Gamma$, is either finite, or $\ZZ$ up to finite index. But all such groups are CAT(0) groups (this follows e.g. from the Bieberbach Theorem, see \cite[II.7, Remark 7.3]{BH}).
\end{proof}

\vspace{0.3cm}

\begin{proof}[Sketch of proof of Theorem \ref{thm: main technical}] Put
	\[
	\CAT=\pr\big (\Cone\cap\HH^n\big ),
	\]
	where $\pr: \Vect\to\Vect$ is the projection from the origin (inducing an isometry $\HH^n\to\BB^n$). Then $\CAT$ is a convex subset of $\BB^n$. The group $\Gamma$ acts on $\CAT$ with a fundamental domain $\Pi_\CAT$, which moreover has finitely many sides. We are going to show that $\CAT$ can be <<decorated>> a bit so that $\Gamma$ acts properly and cocompactly on the resulting CAT(0) space.
	
	Recall that a point $a\in\partial\overline{\BB}^n=\Sph^{n-1}$ is a {\it limit point} of a subgroup $\Gamma\subset\Mob(\BB^n)$ if there is a point $b\in\BB^n$ and a sequence $\{\gamma_i\in\Gamma\}_{i=1}^\infty$ such that $\{\gamma_i b\}_{i=1}^\infty$ converges to $a$. Let $C(\Gamma)$ denote the convex hull of the set of limit points of $\Gamma$ on $\overline{\BB}^n$. Note that this is a closed subset of $\overline{\BB}^n$ \cite[\S 12.1]{Rat}. Put 
	\[
	X=\CAT\cap C(\Gamma),\ \ \ \ \Pi=\Pi_\CAT\cap C(\Gamma).
	\]
\begin{description}
	\item[Step 1] There exists a finite family $U$ of horocusp regions with disjoint closures such that $\Pi\setminus U$ is compact. See paragraph \ref{subsec: step 1} for details.
	
	\item[Step 2] Put $U'=\bigcup_{\gamma\in\Gamma}\gamma(U)$. Step 1 shows that $\Gamma$ acts cocompactly on $X\setminus U'$.
	
	\item[Step 3] Besides, $\Gamma$ acts properly discontinuously on $\BB^n$, hence on $X$ and $X\setminus U'$ too. Since we assume that the kernel of the induced homomorphism
	\[
	\Phi: \Gamma\to\Isom(\BB^n)
	\]
	is finite, it suffices to show that $\Phi(\Gamma)$ acts properly on $\BB^n$. By \cite[Lemma 3.1.1]{Wolf} if $H$ and $K$ are subgroups of a group $G$ with $K$ compact and $G$ locally compact, then $H$ is properly discountinous on $G/K$ if and only if $H$ is discrete in $G$. Now take $G=\Isom(\HH^n)\cong\OO^+(1,n)$, $H=\Phi(\Gamma)$, and $K=\Stab(x)\cong\OO_n(\RR)$, where $x\in\HH^n$. Notice that $\OO^+(1,n)$ is transitive on $\HH^n$ and $\HH^n\cong\OO^{+}(1,n)/\OO_n(\RR)$, see \cite[I.2.24]{BH}.
	
	\item[Step 4] 
	The radii of the horoballs of $U'$ can be decreased such that we obtain a new collection $W$ of open horoballs with disjoint closures and $X\setminus W$ is a CAT(0) space. This is explained in paragraph \ref{subsec: step 4}.
	
	\item[Step 5]
	The action of $\Gamma$ on $X\setminus W$ clearly remains properly discountinuous. It also remains cocompact by Remark \ref{rem: shrinking preserves compactness}. This completes the proof.
\end{description}

\end{proof}

\begin{figure}[h!]
	\centering
	\includegraphics[width=1\linewidth]{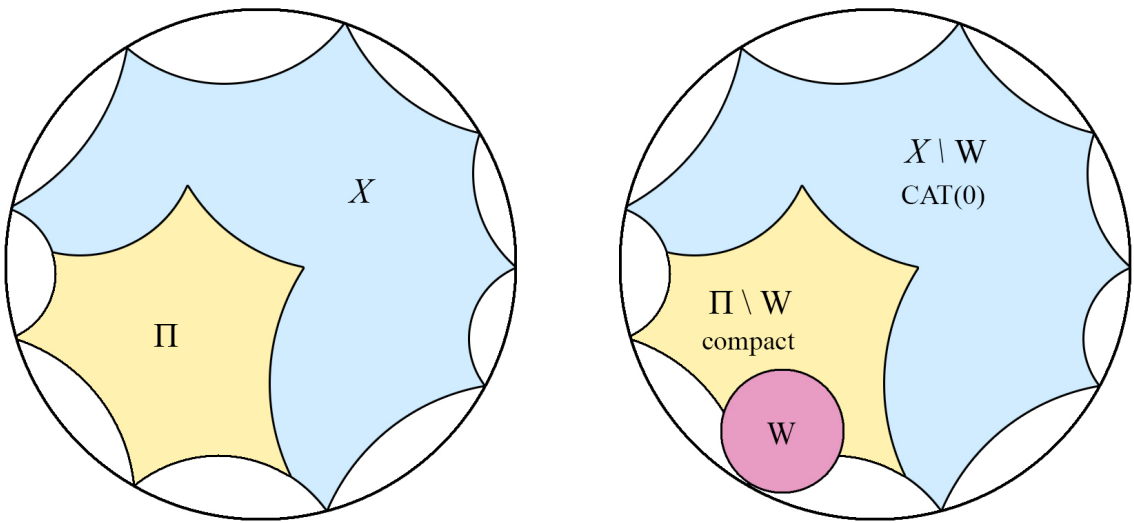}
	\label{pic:acnode crunode}
	\caption{Construction of a CAT(0) space $\XX_M$ (of course, the reader should not take this picture too seriously, as typically we must delete infinitely many horoballs).}
\end{figure}

In the next two sections we explain the difficult parts of the proof.

\subsection{Explanation of Step 4}\label{subsec: step 4}

Let $U=B_1\sqcup\ldots\sqcup B_N$, where $B_i$ are horocusp regions with disjoint closures, constructed in Step 1. By the definition of a horocusp region, for each $i$ the set $\cup_{\gamma\in\Gamma} \gamma(B_i)$ consists of pairwise disjoint balls. One can view the set
\[
U'=\bigcup_{\gamma\in\Gamma}\gamma(U)=\bigcup_{i=1}^N\bigcup_{\gamma\in\Gamma}\gamma(B_i)
\]
constructed in Step 2 as a finite collection of disconnected sets in a metric space. Clearly one can decrease the radii of $B_i$ so that $U'$ is a family of disjoint open horoballs. Denote this resulting family by $U''$. We are now going to show that $U''$ can be shrinked further so that $X\setminus U''$ is a CAT(0) space. First, note that $\HH^n\setminus U''$ is a complete CAT(0) space by the following general fact:

\begin{thm}[{\cite[II.11.27]{BH}}]\label{thm: induced CAT}
	Let $Y\subset\BB^n$ be a subspace obtained by deleting a family of disjoint open horoballs. When endowed with the induced length metric, $Y$ is a complete CAT(0) space.
\end{thm}

Obviously, a convex subset of a CAT(0) space is itself a CAT(0) space when endowed with the induced metric. In view of Theorem \ref{thm: induced CAT}, to conclude that $X\setminus U''$ is a CAT(0) space we need to check its convexity in $\BB^n\setminus U''$ (for the induced length metric), possibly after decreasing radii of $U''$.

First let us mention the description of geodesics in the truncated hyperbolic space $\BB^n\setminus U''$.
\begin{thm}[{\cite[Corollary 11.34]{BH}}]\label{thm: truncated}
	Let $Y$ be as in Theorem \ref{thm: induced CAT}. A path $c: [a,b]\to Y$ parametrized by arc length is a geodesic in $Y$ if and only if it can be expressed as a concatenation of non-trivial paths $c_1,\ldots, c_n$ parametrized by arc length, such that:
	\begin{enumerate}
		\item each of the paths $c_i$ is either a hyperbolic geodesic or else its image is contained in one of the horospheres bounding $Y$ and in that horosphere it is a Euclidean geodesic;
		\item if $c_i$ is a hyperbolic geodesic the the image $c_{i+1}$ is contained in a horosphere and vice versa.
	\end{enumerate}
\end{thm}
Now pick two different points $x,y\in X\setminus U''$. Let $\ell$ be the geodesic of $\BB^n$ joining them. Note that $\ell\subset X$, as $X$ is convex in $\BB^n$. If $\ell\cap U''=\varnothing$, then $\ell\subset X\setminus U''$ and we are done. So, let us assume that $\ell$ intersects $U''$. The geodesic of $X\setminus U''$ is a concatenation of hyperbolic geodesics $\alpha_i$ and Euclidean geodesics $\beta_j$ lying on horospheres. Note that both endpoints of each $\alpha_i$ belong to $X$, hence all $\alpha_i$ lie in $X\setminus U''$ (as $X$ is convex in $\BB^n$ and $\alpha_i$ are hyperbolic geodesics of $\BB^n$). To make sure that $\beta_j$ lie in $X$ we can decrease the radius of each $B_i$ so that the antipodal point of its base point belongs to $X$. This will hold for all $\gamma(B_i)$ in fact, as $X$ is $\Gamma$-invariant. Denote by $W$ the resulting disjoint family of open horoballs. Now all $\beta_j$ are contained in $X$, so the whole geodesic between $x$ and $y$ is contained in $X\setminus W$. So, $X\setminus W$ is a CAT(0) space.


\subsection{Explanation of Step 1}\label{subsec: step 1}

We again need some definitions from metric geometry. A group $G\subset\Mob(\UU^n)$ is called {\it  elementary} if $G$ has a finite orbit in $\widehat{\EE}^n$. An elementary group $G$ is said to be of {\it parabolic type} if $G$ fixes a point on $\widehat{\EE}^n$ and has no other finite orbits.

Let $\Gamma\subset\Mob(\UU^n)$ be a discrete subgroup such that $\infty$ is fixed by a parabolic element of $\Gamma$. Then $\Gamma_\infty=\Stab_\Gamma(\infty)$ is an elementary group of parabolic type. Thus $\Gamma_\infty$ corresponds under Poincar\`{e} extension (see above) to a discrete subgroup of $\Isom(\EE^{n-1})$ \cite[Theorem 5.5.5]{Rat}. We will need the following classical {\bf Bieberbach Theorem}, see e.g. \cite[Theorem 5.4.6, 7.4.2]{Rat}
\begin{theorem*}
	For any discrete subgroup $\Xi\subset\Isom(\EE^n)$ there exists
	\begin{itemize}
		\item a free abelian subgroup $\Xi'\subset\Xi$ of rank $m$ and finite index;
		\item an $\Xi$-invariant $m$-plane $Q\subset\EE^n$ such that $\Xi'$ acts effectively on $Q$ as a discrete group of translations.
	\end{itemize}
	Moreover, the group $\Xi$ is a crystallographic group of $Q$, meaning that $Q/\Xi$ is compact.
\end{theorem*}
Now take $\Xi=\Gamma_\infty$. By the Bieberbach theorem there is a $\Gamma_\infty$-invariant $m$-plane $Q$ of $\EE^{n-1}$ with $Q/\Gamma_\infty$ compact. Denote by $N(Q,\varepsilon)$ the $\varepsilon$-neighborhood of $Q$ in $\EE^n$. Then $N(Q,\varepsilon)$ is $\Gamma_\infty$-invariant. Set
\[
U(Q,\varepsilon)=\overline{\UU}^n\setminus\overline{N}(Q,\varepsilon).
\]
This is an open $\Gamma_\infty$-invariant subset of $\overline{\UU}^n$. It is called a {\it cusped region} for $\Gamma$ based at $\infty$ if for all $\gamma\in\Gamma\setminus\Gamma_\infty$ we have
\begin{equation}\label{eq: precise invariance}
U(Q,\varepsilon)\cap \gamma U(Q,\varepsilon)=\varnothing.
\end{equation}
Viewed in $\BB^n$ model when $m=n-1$, the sets $U(Q,\varepsilon)$ are just horocusp regions based at $\infty$, as defined in Subsection \ref{subsec: hyperbolic prelim}. Let $c\in\widehat{\EE}^{n-1}$ be a point fixed by a parabolic element of a discrete subgroup $\Gamma\subset\Mob(\UU^n)$. A subset $U\subset\overline{\UU}^n$ is a {\it cusped region} for $\Gamma$ based at $c$ if upon conjugating $\Gamma$ so that $c=\infty$, the set $U$ transforms to a cusped region for $\Gamma$ based at $\infty$. A {\it cusped limit point} of $\Gamma$ is a fixed point $c$ of a parabolic element of $\Gamma$ such that there is a cusped region $U$ for $\Gamma$ based at $c$.

Recall that by $L(\Gamma)$ we denoted the set of limit points of $\Gamma$ in $\overline{\BB}^n$, and by $C(\Gamma)$ the convex hull of $L(\Gamma)$ in $\overline{\BB}^n$. Now Step 1 is the content of the following claim\footnote{It is clear that at this step we may assume that $\Gamma$ acts effectively.}.

\begin{prop}\label{prop: key prop}
	Let $\Gamma\subset\Isom(\BB^n)$ be a discrete subgroup, and $Z$ be a $\Gamma$-invariant convex subset of $\BB^n$. Assume that the action of $\Gamma$ on $Z$ has a finitely sided (locally finite) polyhedral fundamental domain $\Pi$. Then there exists a finite union $V$ of horocusp regions with disjoint closures such that $(\Pi\cap C(\Gamma))\setminus V$ is compact in $Z$.
\end{prop}
In what follows $\overline{\Pi}$ will denote the closure of fundamental polyhedron $\Pi$ in $\overline{\BB}^n$. The most importing step in the proof Proposition \ref{prop: key prop} is the following Lemma.
\begin{lem}[cf. {\cite[Lemma 4.10]{Apanasov}}]\label{lem: apanasov lemma}
	Under assumptions of Proposition \ref{prop: key prop} the set $P=\overline{\Pi}\cap L(\Gamma)$ is a finite set of cusped limit points of $\Gamma$. 
\end{lem}

Before proving it, recall that a {\it conical limit (approximation) point} for a subgroup $\Gamma\subset\Mob(\BB^n)$ is a point $a\in\partial\overline{\BB}^n$ such that there exists a point $x\in\BB^n$, a sequence $\{\gamma_i\in\Gamma\}_{i=1}^\infty$, a geodesic ray $\beta\subset\BB^n$ ending at $a$ and $r\in\RR_{>0}$ such that $\{\gamma_i x\}_{i=1}^\infty$ converges to $a$ in $N(\beta,r)$. It is easy to show that this definition does not depend on a choice of $\beta$ and $x$ (see e.g. \cite[Theorem 12.2.2, 12.2.3]{Rat}), so in what follows we will choose $x=0$ and a convenient geodesic ray.

\begin{ex}
	A fixed point $a$ of any loxodromic element $\gamma\in\Gamma$ is always a conical limit point. Indeed, replacing $\gamma$ with $\gamma^{-1}$ we may assume that $a$ is the attractive fixed point. Then for any point $x$ on the axis of $\gamma$ the sequence $\{\gamma^ix\}_{i=1}^\infty$ converges to $a$. 	
\end{ex}

\begin{proof}[Proof of Lemma \ref{lem: apanasov lemma}]
	First, $P$ contains no loxodromic fixed points, as it cannot contain conical limit  points. Indeed, assume that $a\in P$ is a conical limit point. Pick a point $u\in\Pi$. Then there is a geodesic ray $\beta=[u,a)$ in $\Pi$. Let $\{\gamma_i\}_{i=1}^\infty$ be a sequence such that $\gamma_i(0)$ converges to $a$ in some $N(\beta,s)$. The condition $\dist(\gamma_i(0),\beta)<s$ implies that we have $\gamma_i^{-1}(\beta)\cap\overline{B(0,s)}\ne\varnothing$, hence for an infinite sequence $\{\delta_i\}_{i=1}^\infty$ we have $\delta_i(\Pi)\cap\overline{B(0,s)}\ne\varnothing$, which contradicts the local finiteness of $\Pi$.
	
	Suppose $p\in P$. We now work in $\UU^n$ model and may assume that $p=\infty$. As $\Pi$ has finitely many sides, there are finitely many points in $P$ lying in $\Gamma$-orbit of $p$. Let these points be $\gamma_i(\infty)$ where $\gamma_0=\id,\ldots,\gamma_k\in\Gamma$. Put $\Gamma_\infty=\Stab_\Gamma(\infty)$. This group acts by Euclidean isometries of $\EE^{n-1}$. By the Bieberbach theorem there exists a free abelian subgroup $\Gamma^\circ\subset\Gamma_\infty$ of finite index $m$ and rank $r\leqslant n-1$, thus there are $g_1,\ldots,g_m\in\Gamma_\infty$ such that
	\[
	\Gamma_\infty=\Gamma^\circ\cup \Gamma^\circ g_1\cup\ldots\cup \Gamma^\circ g_m.
	\]
	Let $\gamma\in\Gamma$ be an element such that $\infty\in \gamma(P)$. Then for some $\gamma_i$ one has $\gamma\gamma_i\in\Gamma_\infty$. Therefore $\gamma$ lies in one of finitely many sets $\Gamma^\circ g_j\gamma_i^{-1}$. This implies that $\Gamma^\circ$ is not trivial: otherwise a neighborhood of $\infty$ would intersect only finite number of polyhedra $\gamma(\Pi)$, and hence $\infty\notin L(\Gamma)$. 
	
	If $r=n-1$, then all $U(Q,\varepsilon)$ are just horoballs based at $p$. It is not difficult to show that for $\varepsilon$ big enough\footnote{Here we consider horoballs as half-spaces $\{x\in\UU^n:\ x_n>\varepsilon\}$, hence big $\varepsilon$ corresponds to a <<small>> horoball in the Poincare ball model $\BB^n$,} they satisfy $U(Q,\varepsilon)\cap\gamma U(Q,\varepsilon)=\varnothing$ for all $\gamma\notin\Gamma_\infty$, see e.g. \cite[Theorem 3.15]{Apanasov} (note that this theorem is applicable as we proved that $\Gamma^\circ\ne\id$, so $\Gamma$ contains Euclidean translations). In particular, $p$ is a cusped point. Let $r< n-1$, and denote by $W$ a minimal $\Gamma_\infty$-invariant $r$-plane in $\EE^{n-1}$. Recall that $\Gamma^\circ$ is generated by $r$ Euclidean translations by linearly independent vectors $v_i\in\EE^{n-1}$. Let $\EE^r=\langle v_1,\ldots,v_r\rangle$, $\EE^{n-1}=\EE^r\times\EE^{n-1-r}$. Then we may assume that any $\gamma^\circ\in \Gamma^\circ$ has the form
	\begin{equation}\label{eq: Bieber}
	\gamma^\circ(x)=A(x)+v,\ \ \ \text{where}\ A\in\OO_{n-1}(\RR),\ \ \ A=\id\ \text{on}\ \EE^{r},\ \ \ A(\EE^{n-1-r})=\EE^{n-1-r},\ \ \ v\in\EE^r.
	\end{equation}
	The previous paragraph shows that for all $\gamma\notin\Gamma_\infty$ polyhedron $\gamma(\Pi)$ lies in one of finitely many hyperbolic half-spaces $H_s$, $s=0,\ldots,mk$, or in the images $\gamma^\circ(H_s)$ under $\gamma^\circ\in\Gamma^\circ$. Note that $H_s$ are Euclidean-bounded. From (\ref{eq: Bieber}) we see that for some $d>0$ all $\gamma^\circ(H_s)$ lie inside the set
	\[
	\big\{x\in\EE^{n-1}:\ \sum_{i=r+1}^{n-1}x_i^2\leqslant d\big \}.
	\]
	Therefore, we get that all polyhedra $\gamma(\Pi)$, $\gamma\notin\Gamma_\infty$, belong to a bounded Euclidean neighborhood of $W$, i.e. they lie outside some cusped region which satisfies (\ref{eq: precise invariance}) then. Hence, $p$ is a cusped limit point.
\end{proof}

\begin{proof}[Proof of Proposition \ref{prop: key prop}](cf. \cite[Theorem 12.4.5]{Rat})
	
	\begin{itemize}
		\item Let $\overline{\Pi}\cap L(\Gamma)$ consist of cusped limit points $c_1,\ldots,c_m$. Choose a proper (i.e. non-maximal) cusped region $U_i$ for $\Gamma$ based at $c_i$ for each $i$ such that $\overline{U}_1,\ldots, \overline{U}_m$ are disjoint and $\overline{U}_i$ meets just the sides of $\Pi$ incident with $c_i$. Further, let $B_i$ be a horoball based at $c_i$ and contained in $U_i$ such that if $gc_i=c_j$ then $gB_i=B_j$. Then $B_i$ is a proper horocusped region for $\Gamma$ based at $c_i$. 
		\item Put $V=\cup_i B_i$ and
		\[
		K=(\Pi\cap C(\Gamma))\setminus V.
		\]
		As $C(\Gamma)$ is closed in $\overline{\BB}^n$ \cite[\S 12.1]{Rat} (and $\Pi$ is closed in $\BB^n$ by definition), the set $K$ is closed in $\BB^n$. Let us show that it is bounded. Assume the contrary and let $\{x_i\}_{i=1}^\infty$ be an unbounded sequence of points in $K$. By passing to a subsequence, we may assume that $\{x_i\}$ converges to a point $a\in\partial\overline{\BB}^n$. Then $a$ is in the set
		\[
		\overline{\Pi}\cap C(\Gamma)\cap\partial \overline{\BB}^n=\overline{\Pi}\cap L(\Gamma).
		\]
		We conclude that $a=c_j$ for some $j$. 
		\item Let us pass to the upper half-space model $\UU^n$ and conjugate $\Gamma$ so that $a=\infty$. Recall that by construction this is a cusped limit point. The Bieberbach Theorem gives us $Q$ --- a $\Gamma_\infty$-invariant $m$-plane of $\EE^{n-1}$ such that $Q/\Gamma_\infty$ is compact. We claim that there exists $r>0$ such that $L(\Gamma)\subset\overline{N}(Q,r)$. Indeed, take $U(Q,r)$ to be a cusped region based at $a$. Then we have  $U(Q,r)\subset\UU^n\cup O(\Gamma)$, where $O(\Gamma)=\partial\overline{\BB}^n\setminus L(\Gamma)$ is the {\it ordinary set} of $\Gamma$. For the proof of this easy statement see \cite[\S 12.3, Lemma 1]{Rat}. Taking complements in $\overline{\UU}^n$, we get $L(\Gamma)\subset\overline{N}(Q,r)$.
		
		\item Now denote by $p: \UU^n\to\EE^{n-1}$ the vertical projection. Since $C(\Gamma)$ is the intersection of all hyperbolic convex subsets of $\overline{\UU}^n$ containing $L(\Gamma)$, and an $r$-neighborhood of a convex set is convex, one has $C(\Gamma)\subset \overline{p^{-1}(\overline{N}(Q,r))}$, so $\{x_i\}_{i=1}^\infty\subset K\subset C(\Gamma)$ implies
		\[
		\{x_i\}_{i=1}^\infty\subset p^{-1}(\overline{N}(Q,r)).
		\]
		As in \cite[Lemma 2, \S 12.3]{Rat} one can show\footnote{Note that in the cited lemma it is in fact not important that $\Pi$ is a fundamental polyhedron for the action of $\Gamma$ on the whole $\UU^n$.} that $\dist_{\EE^n}(x_i,Q)\to\infty$, so the sequence of $n$-th coordinates of $x_i$ also converges to $\infty$, which means that $x_i$ is in $B_j$ for all sufficiently large $i$ (that was {\it not} clear a priori, since, informally speaking, $x_i$ could converge to $a$ ``tangentially'' to $B_j$, without entering this horoball). But this is is a contradiction, since $K$ is disjoint from $B_j$. Thus $K$ is bounded, hence also compact.
	\end{itemize}
\end{proof}

\begin{rem}\label{rem: shrinking preserves compactness}
	Note that after finding horoballs $B_i$ corresponding to cusped regions $U_i$, we can further shrink them if needed, and the proof of the compactness of $(\Pi\cap C(\Gamma)\setminus V)$ still remains valid. 
\end{rem}

\section{Further examples}\label{sec: examples}

Here we discuss some examples where the Kawamata-Morrison cone conjecture is known to be true.

\begin{ex}[\bf Calabi-Yau pairs]
	Let $X$ be a smooth complex algebraic variety and $\Delta=\sum_i a_i\Delta_i$ be an effective $\RR$-divisor on $X$ with simple normal crossings. Recall that $(X,\Delta)$ is called a {\it klt Calabi-Yau pair} if all $a_i<1$ and $K_X+\Delta$ is numerically trivial. The cone conjecture holds for any klt Calabi-Yau surface \cite[Theorem 3.3]{Totaro}. Here are some interesting examples of klt Calabi-Yau pairs (besides trivial examples with $\Delta=0$).
	\begin{itemize}
		\item {\it Halphen surfaces} of index $m\geqslant 1$, i.e. smooth projective rational surfaces $X$ such that $|-mK_X|$ is a pencil without base points and fixed components. It can be shown that $X$ is a Halphen surface if and only if there exists an irreducible pencil of curves of degree $3m$ in $\PP^2$ with 9 base points of multiplicity $m$
		(maybe infinitely near) such that $X$ is the blow-up of these 9 base points and $|-mK_X|$ is the strict
		transform of this pencil. This is equivalent to the existence on $X$ of relatively minimal elliptic fibration  with no multiple fibre if $m=1$	and a unique multiple, of
		multiplicity $m$, otherwise \cite[Proposition 2.2]{Cantat-Dolgachev};
		\item {\it Coble surfaces}, i.e. smooth projective rational surfaces $X$ such that $|-K_X|=\varnothing$, but $|-2K_X|$ consists of a single irreducible curve $C$, and $K_X^2=-1$. The blowing down of a $(-1)$-curve on a Coble surfaces gives a Halphen surface of index 2. Conversely, the blow-up of a singular point of an irreducible non-multiple fiber of a Halphen surface of index 2 is a Coble surface \cite[Proposition 3.1]{Cantat-Dolgachev};
	\end{itemize}
	Although Halphen and Coble surfaces are rational by definition, they are still relevant for present discussion. Indeed, one has a representation $\Aut(X)\to\OO(\NS(X),q)$, where $q$ is the intersection form on $\NS(X)\cong\Pic(X)$ of signature $(1,\rho(X)-1)$ by Hodge index Theorem. The Kawamata-Morrison conjecture for $\Aut(X)$ holds by \cite{Totaro}.
	
	Let $X$ be a Halphen surface. By \cite[Remark 2.11]{Cantat-Dolgachev} there exists a homomorphism 
	\[
	\rho:\ \Aut(X)\to\PGL_2(\CC)
	\]
	with finite image such that $\ker\rho$ is an extension of a free abelian group of rank $\leqslant 8$ by a cyclic group of order dividing 24. In particular $\Aut(X)$ is finitely generated and satisfies Tits' alternative.
	
	Automorphism groups of Coble surfaces are also quite interesting. Automorphism group of a general Coble surface is isomorphic to the automorphism group of a general Enriques surface, and the latter are known to have a rich structure. For example, in \cite[Exemples 3.8]{Deserti} the authors gives an example of a Coble surface whose automorphism groups contains $(\ZZ^8)^{*10}$.
\end{ex}

\begin{ex}[\bf Wehler varieties]
	Another example of Calabi-Yau manifolds for which the Kawamata-Morrison conjecture holds is given by so-called {\it Wehler varieties}. These are smooth hypersurface $X$ of degree $(2,2,\ldots,2)$ in $(\PP^1)^{N+1}$. The $N+1$ projections $\pi_k: X\to(\PP^1)^N$ which are obtained by forgetting one coordinate are ramified double coverings. So, there are birational transformations $\iota_k: X\dashrightarrow X$ that permute the two points in the fibers of $\pi_k$. This gives a homomorphism from the universal Coxeter group $\UC(N+1)$ of rank $N+1$ to $\Bir(X)$:
	\[
	\Psi: \UC(N+1)=\underbrace{\ZZ/2*\ldots*\ZZ/2}_\text{$N+1$ times}\to\Bir(X).
	\]
	In \cite{Cantat-Oguiso} it was shown that for $N\geqslant 3$ and generic $X$ one has $\Aut(X)=\id$, and $\Psi$ is in fact an isomorphism (in particular, $\Bir(X)$ is finitely presented and satisfies Tits' alternative). Recall that given a Coxeter system $(W,\{s_j\}_{j=1}^N)$ one can construct an $N$-dimensional vector space $V_{N}$ of its geometric representation, with a quadratic form $b_{N}$ preserved by $W$. The group $W$ also preserves the so-called Tits cone $\Tits$, that contains an explicit subcone $\DTits$ which is a fundamental domain for the action of $W$ on $\Tits$. It then turns out that the action of $\Bir(X)$ on the Neron-Severi group $\NS(X)$ is conjugate to the	geometric representation $V_{N+1}$ of $\UC(N+1)$. Moreover, the form $b_{N+1}$ has signature $(1,N)$ and there exists a linear isomorphism $V_{N+1}\to\NS(X)_\RR$ such that the fundamental domain of $\UC(N+1)$ is mapped onto $\overline{\Amp}(X)$ and the Tits cone is mapped onto $\overline{\Mov}^e(X)$. So, our method here applicable only post factum, as an explicit description of $\Bir(X)$ is already available. However it would be interesting to know if there is an apriori explanation\footnote{Note that one of the things to explain here is the geometric meaning of the form $b_{N+1}$ for $N\geqslant 3$.} of any finiteness properties for $\Bir(X)$ in spirit of the following conditional statement (which follows from Proposition \ref{prop: kernel is finite} and Theorem \ref{thm: main technical}):
	
	\begin{conditional}\label{thm: general CY}
		Let $M$ be a projective Calabi-Yau manifold such that there exists a hyperbolic $\Aut(M)$-invariant (resp. $\Bir(M)$-invariant) quadratic form $q$ on $\NS(M)$. Assume that the Kawamata-Morrison cone conjecture (resp. its birational version) is true for $M$. Then $\Aut(M)$ (resp. $\Bir(M)$) has all the properties listed in Theorems A, B and C.
	\end{conditional}

\end{ex}

\hfill

{\bf Acknowledgements:} Authors thank Misha Verbitsky for discussions, indicating idea of the proof of the Kawamata-Morrison conjecture in $b_2=5$ case. Both authors are Young Russian Mathematics award winners and would like to thank its sponsors and jury.

\def\bibindent{2.5em}

\hfill

{\small
\noindent
{\bf Nikon Kurnosov}\\
Department of Mathematics,\\
University of Georgia,\\
Athens, GA, USA, 30602\\
also:
Laboratory of Algebraic Geometry,\\
National Research University HSE,\\
Department of Mathematics,\\
6 Usacheva Str., 119048, Moscow, Russia\\
{\tt nikon.kurnosov@gmail.com}

\hfill

\noindent {\bf Egor Yasinsky}\\
Universit\"{a}t Basel \\
Departement Mathematik und Informatik\\
Spiegelgasse 1, 
4051 Basel \\
{\tt  yasinskyegor@gmail.com}.

\end{document}